\documentclass{amsart}
\usepackage[margin=3cm]{geometry}
\usepackage{amsmath,amsfonts,amssymb,amsthm}
\usepackage{graphics}
\usepackage{epsfig, esint}
\usepackage{graphics,xcolor}
\usepackage{paralist}

\RequirePackage[colorlinks,citecolor=blue,urlcolor=blue]{hyperref}

   \definecolor{labelkey}{gray}{.8}

\providecommand{\dx}{\, \mathrm{d} x}
\providecommand{\dy}{\, \mathrm{d} y}

\providecommand{\R}{\mathbb{R}}
\providecommand{\N}{\mathbb{N}}

\providecommand{\hom}{{\rm hom}}

\newcommand{\e}{\varepsilon}
\newcommand{\Z}{\mathbb Z}

\newcommand{\step}[1]{\medskip\noindent\textbf{Step #1. }}

\newcommand{\ignore}[1]{}

\newtheorem{theorem}{Theorem}

\newtheorem{proposition}[theorem]{Proposition}
\newtheorem{remark}[theorem]{Remark}
\newtheorem{lemma}[theorem]{Lemma}
\newtheorem{corollary}[theorem]{Corollary}
\newtheorem{assumption}[theorem]{Assumption}

\author{Matthias Ruf}
\address[Matthias Ruf]{Institut f\"ur Mathematik, Universit\"at Augsburg, Universit\"atsstra\ss e 14, 86159 Augsburg, Germany}
\email{matthias.ruf@uni-a.de}

\author{Mathias Sch\"affner}
\address[Mathias Sch\"affner]{Institut f\"ur Mathematik, MLU Halle-Wittenberg, Theodor-Lieser-Stra\ss e 5, 06120 Halle (Saale), Germany}
\email{mathias.schaeffner@mathematik.uni-halle.de}

\title[Homogenization in nonlinear elasticity]{Upper bounds for the homogenization problem in nonlinear elasticity: the incompressible case}

\begin{document}
	
	\begin{abstract}
		We consider periodic homogenization of hyperelastic models incorporating incompressible behavior via the constraint $\det(\nabla u)=1$. We show that the 'usual' homogenized integral functional $\int W_{\rm hom}(\nabla u)\,dx$, where $W_{\rm hom}$ is the standard multicell-formula of non-convex homogenization restricted to volume preserving deformations, yields an upper bound for the $\Gamma$-limit as the scale of periodicity tends to zero.

	\end{abstract}
	
\maketitle

{\small
	\noindent
	
	\noindent
}	

\section{Introduction and main result}
In this note we provide partial results towards periodic homogenization of nonlinear elastic incompressible materials, modeled by integral functionals of the form
\begin{equation}\label{intro:int}
F_\e:\, W^{1,1}(\Omega)^d\to[0,+\infty],\qquad F_\e(u,\Omega)=\int_\Omega W(\tfrac{x}\e,\nabla u(x))\dx.
\end{equation}
Here, $\Omega\subset \R^d$ with $d\geq2$ is a bounded domain and the stored elastic energy density $W:\R^d\times \R^{d\times d}\to[0,+\infty]$ is assumed to be $Y=(-\frac12,\frac12)^d$-periodic in the first variable. Homogenization of integral functionals \eqref{intro:int} is well established: Following earlier results by Marcellini \cite{Marcellini} for convex integrands, it is shown in the seminal contributions by Braides \cite{Braides85} and M\"uller \cite{Mueller87} that under standard $p$-growth assumptions, that is,
%
\begin{equation}\label{p:growth}
\exists p>1\,\exists c\geq 1:\quad \frac{1}{c}|F|^p-c\leq W(y,F)\leq c(1+|F|)^p\quad\mbox{for all $F\in\R^{d\times d}$ and a.e.\ $y\in\R^d$,}
\end{equation} 
the sequence $(F_\e)$ $\Gamma$-converges as $\e\to0$ towards an autonomous integral functional
\begin{equation*}
	F_{\rm hom}(u)=\int_{\Omega}W_{\rm hom}(\nabla u)\dx,
\end{equation*}
where the homogenized integrand $W_{\rm hom}$ is given by a multi-cell formula
\begin{equation}\label{intro:multicell}
\overline W(F)=\inf_{k\in\N}\inf_{\varphi\in W^{1,p}_{0}(kY)}\fint_{kY} W(x,F+\nabla\varphi)\dx.
\end{equation}
While the homogenization formula \eqref{intro:multicell} captures non-trivial effects in nonlinear elasticity, such as buckling (see e.g. \cite[Theorem 4.3]{Mueller87}), and it is expected that \eqref{intro:multicell} should apply to physically sound stored-energy functions (see \cite{LFL22}), it was already remarked in \cite{Mueller87} that the above mentioned $\Gamma$-convergence results do not apply directly to nonlinear elasticity for the growth conditions in \eqref{p:growth} being too restrictive. Realistic elastic energy densities $W$ should satisfy  
\begin{equation}\label{ass:OP}
W(x,F)=+\infty\quad\mbox{ if $\det(F)<0$ and }\qquad W(x,F)\to+\infty\quad\mbox{as}\quad \det(F)\to0,
\end{equation}
in order to rule out interpenetration of matter and to ensure that it takes infinite energy to squeeze material to zero volume. Clearly, \eqref{ass:OP} is incompatible with \eqref{p:growth}. While there are several homogenization results under relaxed versions of  the growth conditions \eqref{p:growth}, see e.g.\ the textbooks \cite{JKO,BD98} or more recent contributions \cite{AM11,ACM17,DG_unbounded,RS24}, there is no $\Gamma$-convergence result which applies to growth conditions of the form \eqref{ass:OP}. To the best of our knowledge, \cite{NS18,NS19} contain the only rigorous homogenization results that apply directly to nonlinear elasticity, which are however restricted to small loads.

%
%

\smallskip

In this note, we provide a nontrivial upper bound for the $\Gamma$-limit of \eqref{intro:int} in the setting of \textit{incompressible} elasticity, that is, we impose an extreme version of \eqref{ass:OP} namely that $W(x,F)=+\infty$ if $\det(F)\neq1$. More precisely, we assume that

\begin{assumption}\label{ass}
Let $W:\R^d\times \Sigma\to[0,+\infty)$,  where $\Sigma:=\{F\in \R^{d\times d}\,:\,\det(F)=1\}$, be a Carath\'eodory-function such that $W(\cdot,F)$ is $Y$-periodic for every $F\in\Sigma$. Moreover, set $W(y,F)=+\infty$ for all $F\in \R^{d\times d}\setminus \Sigma$ and all $y\in Y$. Define $W_{\hom}$ via the multi-cell formula
\begin{equation}\label{eq:defW_hom}
	W_{\rm hom} (F):=\inf_{k\in\N} W_{\rm hom}^{(k)}(F),\quad\mbox{where}\quad W_{\rm hom}^{(k)}(F):=\inf_{\varphi\in W_0^{1,\infty}(kY)^d}\fint_{kY}W(y,F+\nabla  \varphi(y))\dy.
\end{equation}
We assume that there exists $c\geq1$ such that for a.e. $x\in\R^d$ and all $F,G\in\Sigma$
\begin{align}
W(x,FG)&\leq c(1+W(x,F))(1+W(x,G)),\label{eq:submultiplicative}
\\
\frac{1}{c}|F|-c&\leq W(x,F)\leq c(\,W_{\rm hom}(F)+1).\label{eq:W_hom_comparison}
\end{align}
Finally, we assume that $W_{\rm hom}$ is finite on $\Sigma$. 
\end{assumption}
\begin{remark}\label{r.onAss}
Assumption~\ref{ass} covers for instance periodic integrands with the following growth conditions:
\begin{enumerate}[(I)]
	\item $p$-growth: $\frac{1}{c}|F|^p-c\leq W(x,F)\leq c(|F|^p+1)$ for all $F\in\Sigma$,
	\item dependence on cofactors: $\frac{1}{c}\left(|F|^p+|{\rm adj}F|^q\right)-c\leq W(x,F)\leq c\left(|F|^p+|{\rm adj}F|^q\right)+c$ for all $F\in\Sigma$.
\end{enumerate}	
Here $p,q\in [1,+\infty)$ are arbitrary growth exponents and ${\rm adj}F$ denotes the adjunct matrix of $F$. For a proof see Lemma~\ref{L:pqexample} below.
\end{remark}

The main result of this work is the following
\begin{theorem}\label{thm:upperbound}
	Let $\Omega\subset\R^d$ be a bounded, open set with Lipschitz boundary and let $W:\R^d\times\R^{d\times d}\to [0,+\infty]$ satisfy Assumption \ref{ass}. Then for every $u\in W^{1,1}(\Omega)^d$ there exists a family $(u_{\e})_{\e>0}\subset u+W^{1,1}_0(\Omega)^d$ such that $u_{\e}\to u$ in $L^1(\Omega)^d$ and 
	\begin{equation}\label{eq:mainrecovery}
	\limsup_{\e\to 0}\int_{\Omega}W(\tfrac{x}{\e},\nabla u_{\e}(x))\dx\leq \int_{\Omega}W_{\rm hom}(\nabla u(x))\dx,
	\end{equation}
	where $W_{\rm hom}$ is defined in \eqref{eq:defW_hom}.	The map $W_{\rm hom}$ is finite only on $\Sigma$, where it is also continuous. Moreover, $W_{\rm hom}$ is quasiconvex in the sense that
	\begin{equation}\label{eq:defquasiconvex}
	W_{\rm hom}(F)\leq \fint_{O}W_{\rm hom}(F+\nabla\varphi(x))\dx\quad\text{ for all }\varphi\in W^{1,\infty}_0(O)^d,
	\end{equation}
	all $F\in\R^{d\times d}$ and all bounded, open sets $O\subset\R^d$.
\end{theorem}
The above theorem shows that a variant of the multi-cell formula of Braides and M\"uller yields an upper bound for the $\Gamma\hbox{-}\limsup$ of the functionals $F_{\e}$ in the incompressible case. Note that the sole difference between $\overline W$ and $W_{\rm hom}$ (see \eqref{intro:multicell} and \eqref{eq:defW_hom}) is that in the former we take the infimum over all $W^{1,p}$-competitors and in the latter over $W^{1,\infty}$-competitors. Clearly, the formulas \eqref{intro:multicell} and \eqref{eq:defW_hom} are equivalent if the energy density satisfies \eqref{p:growth}. In the setting of Theorem~\ref{thm:upperbound} this equivalence is not clear and there might be a Lavrientiev gap. Unfortunately, we are not able to provide the corresponding $\liminf$ inequality. In Proposition~\ref{thm:lowerbound}, under a stronger lower bound on $W$ we provide a non-trivial lower bound on the $\Gamma$-limit based on truncations that lead to an integrand that is finite only on $\Sigma$, where it is also continuous. Moreover, like $W_{\rm hom}$ it is quasiconvex. We further show that in the polyconvex case there is a commutation of truncation and homogenization for the finite-cell formulas. However, in general we do not know if the lower bound equals the upper bound of Theorem~\ref{thm:upperbound}.

\medskip

Let us now comment on the proof of Theorem~\ref{thm:upperbound}. The classical construction for the upper bound of the Gamma-limit, see e.g.\ \cite{Mueller87}, is based on density results. Indeed, by definition of \eqref{intro:multicell} (or \eqref{eq:defW_hom}) it is straightforward to construct a recovery sequence for affine or piecewise affine functions $u$. In order to pass from piecewise affine to general $W^{1,1}(\Omega)^d$-functions it is important that the energy $F_{\rm hom}$ is continuous with respect to the approximation. In particular we would need to approximate a volume preserving deformation by piecewise affine volume preserving deformations. To the best of our knowledge, this is an open problem. 

Here, we take a different approach which is strongly inspired by the relaxation result \cite{CoDo}, where Conti and Dolzmann considered the relaxation problem for autonomous integrands under constraints on the determinant (see also \cite{CF} for an extension of \cite{CoDo} to non-autonomous problems). A key idea in \cite{CoDo} is that instead of \textit{adding} oscillations to the macroscopic deformation one locally use the \textit{composition} of the macroscopic deformation with maps with oscillating gradients. In the context of homogenization of elliptic equations this is sometimes referred to as harmonic coordinates. More precisely, let $u$ be a function which is close to an affine function with gradient $F\in\Sigma$ and consider the composition $u\circ (\cdot+F^{-1}\e\varphi(\tfrac{\cdot}\e))$, where $\varphi$ is a competitor for \eqref{eq:defW_hom}. The chain rule formally yields $$\nabla u(\cdot+F^{-1}\varphi(\tfrac\cdot\e))(I+F^{-1}\nabla\varphi(\tfrac\cdot\e)).$$ If $u$ was affine with gradient $F$, then this formula reduces to $F+\nabla\varphi(\tfrac\cdot\e)$, but the big advantage is that the multiplicative structure of the gradient preserves the constraints on the determinant. We mention here that in \cite{CoDo} also relaxation of compressible models is studied with similar techniques. Unfortunately we are not able to extend the method to the problem of homogenization of compressible elasticity, see Remark~\ref{r.compressible}.

This note is organized as follows: in Section \ref{sec:prelim} we gather some technical results for $W_{\rm hom}$ defined in \eqref{eq:defW_hom}. In Section~\ref{sec:proofs}, we prove Theorem~\ref{thm:upperbound}. Finally, we discuss in Section~\ref{sec:LB} a lower bound on the $\Gamma$-limit of $F_\e$ via truncation of the integrand.

\section{Preliminary results and properties of \texorpdfstring{$W_{\rm hom}$}{Whom}}\label{sec:prelim}


First, we provide some bounds on the multi-cell formula and establish a useful alternative characterization of $W_{\rm hom}$, see \eqref{eq:defW_hom}.
\begin{lemma}\label{L:generalcellformula}
Suppose that Assumption~\ref{ass} is satisfied. For all $F\in\Sigma$, it holds
\begin{equation}\label{ass:Whomgrowth}
	\frac{1}{c}|F|-c\leq W_{\rm hom}(F)
\end{equation}
where $c$ is given by \eqref{eq:W_hom_comparison}, while it holds $W_{\rm hom}(F)=+\infty$ if $F\notin \Sigma$. Moreover, for all $F\in\Sigma$ and every bounded, open set $O\subset\R^d$ with Lipschitz-boundary it holds that
\begin{equation}\label{limit_whombyO}
W_{\rm hom} (F)=\lim_{k\to+\infty}\inf_{\varphi\in W_0^{1,\infty}(O)^d}\fint_{O}W(ky,F+\nabla  \varphi(y))\dy.
\end{equation}
\end{lemma}

\begin{proof}
We first show that $W_{\rm hom}(F)=+\infty$ if $F\notin \Sigma$, that is $\det(F)\neq 1$. Let $k\in\N$ and $\varphi\in W^{1,\infty}_0(kY)^d$. Since the determinant is a null Lagrangian, we have
\begin{align*}
	0<|\det F-1|=\biggl|\fint_{kY} \det F-1\dy\biggr|=\biggl|\fint_{kY} \det (F+\nabla\varphi)-1\dy\biggr|\leq \fint_{kY}|\det (F+\nabla\varphi)-1|\dy.
\end{align*} 
Hence, there exists a set of positive measure $A\subset kY$ such that $\det (F+\nabla\varphi)\neq 1$ on $A$ and thus 
$$
\fint_{kY}W(y,F+\nabla\varphi(y))\dy=+\infty,
$$
which implies, by the arbitrariness of $k\in\N$ and $\varphi\in W_0^{1,\infty}(kY)^d$, that $W_{\rm hom}(F)=+\infty$ (see \eqref{eq:defW_hom}).
Thus, it suffices to consider $F\in\Sigma$. Next we show the lower bound \eqref{ass:Whomgrowth} on $W_{\rm hom}$. For every $k\in\N$ and $\varphi\in W^{1,\infty}_0(kY)^d$ we find that
\begin{equation}\label{est:keylbwhom}
	c^{-1}|F|-c= c^{-1}\left|\fint_{kY}F+\nabla\varphi\dy\right|-c\leq \fint_{kY}c^{-1}|F+\nabla\varphi|-c\dy\overset{\eqref{eq:W_hom_comparison}}{\leq} \fint_{kY}W(y,F+\nabla\varphi)\dy.
\end{equation}
Taking the infimum over $\varphi$ and then over $k\in\N$, from Definition \eqref{eq:defW_hom} we deduce the lower bound $c^{-1}|F|-c\leq W_{\rm hom}(F)$. 

Finally, we consider a bounded open set $O\subset\R^d$ with Lipschitz-boundary and show \eqref{limit_whombyO}. This will be done by gradual approximation. By a change of variables we have that
\begin{equation*}
	a_k(O):=\inf_{\varphi\in W^{1,\infty}_0(O)^d}\fint_{O}W(ky,F+\nabla\varphi(y))\dy=\inf_{\varphi\in W^{1,\infty}_0(kO)^d}\fint_{kO}W(y,F+\nabla\varphi(y))\dy.
\end{equation*}
We first consider the case $O=Y$ and show that $a_k(Y)\to W_{\rm hom}(F)$. To this end, let $\delta>0$ and take $k_0\in\N$ and $\varphi\in W^{1,\infty}_0(k_0Y)^d$ such that
\begin{equation*}
	W_{\rm hom}(F)\geq \fint_{k_0Y}W(y,F+\nabla\varphi(y))\,\dy-\delta.
\end{equation*}
Consider $k\gg k_0$ and extend $\varphi$ by $k_0Y$-periodicity. We define $\widetilde{\varphi}\in W^{1,\infty}_0(kY)^d$ via
\begin{equation*}
	\widetilde{\varphi}(y)=\begin{cases}
		\varphi(y) &\mbox{if $y\in z+k_0Y\subset kY$ for some $z\in k_0\Z^d$,} \\
		0 &\mbox{otherwise.}
	\end{cases}
\end{equation*}
Then by $k_0Y$-periodicity of $\varphi$ and $W(\cdot,F)$ and the upper bound in \eqref{eq:W_hom_comparison}, we have that
\begin{align*}
	a_k(Y)&\leq \fint_{kY}W(y,F+\nabla\widetilde{\varphi}(y))\,\dy
	\\
	&\leq \frac{1}{k^d}\sum_{\substack{z\in k_0\Z^d\\ z+k_0Y\subset kY}}\int_{z+k_0Y}W(y,F+\nabla\varphi(y))\dy
	\\
	&\quad+\frac{1}{k^d}\sum_{\substack{z\in k_0\Z^d\\ z+k_0Y\cap\partial kY\neq\emptyset}}\int_{z+k_0Y}c(W_{\rm hom}(F)+1)\dx
	\\
	&\leq \frac{(k/k_0)^d}{k^d}\int_{k_0Y}W(y,F+\nabla\varphi(y))\dy
	\\
	&\quad+C\frac{(k+k_0)^{d-1}k_0}{k^d}k_0^d(W_{\rm hom}(F)+1).
\end{align*}
Letting $k\to +\infty$ we find that
\begin{equation*}
	\limsup_{k\to +\infty}a_k(Y)\leq \fint_{k_0Y}W(y,F+\nabla\varphi(y))\dy\leq W_{\rm hom}(F)+\delta.
\end{equation*}
As $\delta>0$ was arbitrary and clearly $a_k(Y)\geq \inf_{n\in\N}a_n(Y)=W_{\rm hom}(F)$, we showed that
\begin{equation}\label{eq:convergenceoncubes}
	W_{\rm hom}(F)=\lim_{k\to +\infty}a_k(Y).
\end{equation}
Now fix an arbitrary bounded, open set $O$ with Lipschitz-boundary. Using the same construction as above with $kY$ replaced by $kO$ and noting that
\begin{equation*}
	\lim_{k\to +\infty}\frac{1}{k^d|O|}\#\{z\in k_0\Z^d:z+k_0Y\cap\partial (k O)\neq\emptyset\}=0
\end{equation*}
due to the Lipschitz-regularity of $\partial O$, we deduce in a similar way that
\begin{equation}\label{eq:limsupanyopen}
	\limsup_{k\to +\infty}a_k(O)\leq W_{\rm hom}(F).
\end{equation}
In order to obtain the reverse inequality for the $\liminf$, let $Q=(-N/2,N/2)^d=NY$ be a large cube that contains $\overline{O}$. Joining any two functions $\varphi_1\in W^{1,\infty}_0(O)^d$ and $\varphi_2\in W^{1,\infty}_0(Q\setminus\overline{O})$ to a function $\varphi\in W^{1,\infty}_0(Q)^d$ one obtains the subadditivity property
\begin{equation*}
	|kQ|a_k(Q)\leq |kO|a_k(O)+|k(Q\setminus\overline{O})|a_k(Q\setminus\overline{O}),
\end{equation*}
where we used that $|\partial O|=0$. Rearranging terms and dividing by $|k O|$, we find that
\begin{align*}
	\liminf_{k\to +\infty}a_k(O)&\geq\liminf_{k\to +\infty}\frac{|Q|}{|O|}a_k(Q)-\limsup_{k\to +\infty}\frac{|Q\setminus\overline{O}|}{|O|}a_k(Q\setminus\overline{O})
	\\
	&\overset{\eqref{eq:limsupanyopen}}{\geq} \liminf_{k\to +\infty}\frac{|Q|}{|O|}a_{Nk}(Y)-\frac{|Q\setminus\overline{O}|}{|O|}W_{\rm hom}(F)\overset{\eqref{eq:convergenceoncubes}}{=}\frac{|Q|}{|O|}W_{\rm hom}(F)-\left(\frac{|Q|}{|O|}-1\right)W_{\rm hom}(F)
	\\
	&=W_{\rm hom}(F).
\end{align*}
This concludes the proof of the lemma.

\end{proof}

In the next lemma, we directly show that $W_{\rm hom}$ is rank-one convex and thus continuous on $\Sigma$.
\begin{lemma}\label{L:Whom:cont}
The function $W_{\rm hom}:\R^{d\times d}\to[0,+\infty]$ is rank-one convex. Hence, as $W_{\rm hom}$ is finite on $\Sigma=\{F\in\R^{d\times d}:\,\det(F)=1\}$ it is continuous on $\Sigma$.
\end{lemma}

\begin{proof}
We first show that $W_{\rm hom}$ is rank-one convex, following the closely related arguments of \cite[Lemma 4.1]{CoDo}. Fix $A,B\in \Sigma$ with ${\rm rank}(A-B)=1$ and $\lambda \in[0,1]$. Moreover, set $F=\lambda A+(1-\lambda)B$. Appealing to \cite[Theorem 2.1]{Conti08}, we find a finite set $K\subset \Sigma$ such that for every $\delta>0$ the following is true: there exists a polyhedron $\Omega_P$ and $u\in W^{1,\infty}(\Omega_P)^d$ satisfying
\begin{align}\label{L:2:eq1}
u(x)=Fx\mbox{ on $\partial \Omega_P$},\quad \nabla u\in K\mbox{ a.e.\ in $\Omega_P$,}\quad |\{\nabla u(x)\notin \{A,B\}\}|\leq \delta|\Omega_P|.
\end{align}
We claim that there exists $c=c(A,B)>0$ such that 
\begin{equation}\label{eq:measurebound}
|\{\nabla u=A\}|\leq (\lambda+c\delta)|\Omega_P|\quad\mbox{and}\quad|\{\nabla u=B\}|\leq (1-\lambda+c\delta)|\Omega_P|.
\end{equation}
Indeed, due to the boundary condition satisfied by $u$ we have that
\begin{equation*}
\lambda A+(1-\lambda)B=\fint_{\Omega_P}F\dx=\fint_{\Omega_P}\nabla u\dx=A\frac{|\{\nabla u=A\}|}{|\Omega_P|}+B\frac{|\{\nabla u=B\}|}{|\Omega_P|}+\frac{1}{|\Omega_P|}\int_{\{\nabla u\notin\{A,B\}\}}\nabla u\dx.
\end{equation*}
Appealing to \eqref{L:2:eq1}, we can bound the norm of the last integral by $\delta\max_{\xi\in K}|\xi|$. Moreover, since $A,B$ are invertible and ${\rm rank}(A-B)=1$, these matrices are linearly independent and by standard linear algebra we find a matrix $C\in\R^{d\times d}$ with norm $1$ and such that $\langle A, C\rangle=0$ and $\langle B, C\rangle\neq 0$. Taking the scalar product of the above equality with the matrix $C$ we find that
\begin{equation*}
	|\{\nabla u=B\}|\leq (1-\lambda+\delta\langle B,C\rangle^{-1}\max_{\xi\in K}|\xi|)|\Omega_P|.
\end{equation*}
A similar argument gives the claimed bound on the measure of $\{\nabla u=A\}$, so that \eqref{eq:measurebound} holds. 
Finally, still due to \cite[Theorem 2.1]{Conti08} the sets $\{\nabla u=A\}$, $\{\nabla u=B\}$ and $\{\nabla u \notin\{A,B\}\}$ are the union of simplices $\omega_j^A$, $j=1,\dots,N_A$, $\omega_j^B$, $j=1,\dots,N_B$ and $\omega_j^C$, $j=1,\dots,N_C$ and on each of these simplices $u$ is affine. For each $k\in\mathbb N$ and $j\in\{1,\dots,N_A\}$, we find $\varphi_{j,k}^A \in W_0^{1,\infty}(\omega_j^A)^d$ such that
\begin{equation}\label{L:whomcontim:ub:A}
\fint_{\omega_j^A}W(ky,A+\nabla \varphi_{j,k}^A)\dy\leq\inf_{\varphi\in W_0^{1,\infty}(\omega_j^A)^d}\fint_{\omega_j^A}W(ky,A+\nabla \varphi)\dy+\delta.
\end{equation}
Analogously, we find for each $k\in\mathbb N$ and $j\in\{1,\dots,N_B\}$, functions $\varphi_{j,k}^B \in W_0^{1,\infty}(\omega_j^B)^d$ such that
\begin{equation}\label{L:whomcontim:ub:B}
\fint_{\omega_j^B}W(ky,B+\nabla \varphi_{j,k}^B)\dy\leq\inf_{\varphi\in W_0^{1,\infty}(\omega_j^B)^d}\fint_{\omega_j^B}W(ky,B+\nabla \varphi)\dy+\delta.
\end{equation}
Extending $\varphi_{j,k}^A$ and $\varphi_{j,k}^B$ by zero and using $u(x)=F(x)$ on $\partial \Omega_P$, we can define $w_k\in W_0^{1,\infty}(\Omega_P)^d$ by
\begin{equation}
w_k(x)=u(x)-Fx+\sum_{j=1}^{N_A}\varphi_{j,k}^A(x) +\sum_{j=1}^{N_B}\varphi_{j,k}^B(x).
\end{equation} 
In view of Lemma~\ref{L:generalcellformula}, one the one hand we have that
\begin{equation}\label{L:whomcontim:lb}
W_{\rm hom}(F)\leq \liminf_{k\to+\infty} \fint_{\Omega_P}W(ky,F+\nabla w_k)\dy\,.
\end{equation}
On the other hand we can write
\begin{align}
\fint_{\Omega_P}W(ky,F+\nabla w_k)\dy=&\frac1{|\Omega_P|}\sum_{j=1}^{N_A}\int_{\omega_j^A}W(ky,A+\varphi_{j,k}^A)\dy+\frac1{|\Omega_P|}\sum_{j=1}^{N_B}\int_{\omega_j^B}W(ky,B+\varphi_{j,k}^B)\dy\notag\\
&+\frac{1}{|\Omega_P|} \sum_{j=1}^{N_C}\int_{\omega_{j}^C}W(ky,G_j)\dy,
\end{align}
for some $G_j\in K$ for all $j=1,\dots,N_C$. Estimates \eqref{eq:measurebound}, \eqref{L:whomcontim:ub:A}, \eqref{L:whomcontim:ub:B} and Lemma~\ref{L:generalcellformula} yield
\begin{align}\label{L:whomcontim:ub}
&\limsup_{k\to+\infty}\fint_{\Omega_P}W(ky,F+\nabla w_k)\dy\notag\\
\leq&\frac1{|\Omega_P|}\sum_{j=1}^{N_A}|\omega_j^A|(W_{\rm hom}(A)+\delta)+\frac1{|\Omega_P|}\sum_{j=1}^{N_B}|\omega_j^B|(W_{\rm hom}(B)+\delta)+\frac{1}{|\Omega_P|} \sum_{j=1}^{N_C}|\omega_{j}^C|\fint_YW(y,G_j)\dy\notag\\
\leq&(\lambda+c\delta)(W_{\rm hom}(A)+\delta)+(1-\lambda+c\delta)(W_{\rm hom}(B)+\delta)+\delta \max_{G\in K}\fint_YW(y,G)\dy.
\end{align}
Combining \eqref{L:whomcontim:lb}, \eqref{L:whomcontim:ub}, and the upper bound on $W$ via $W_{\rm hom}$ we find
$$
W_{\rm hom}(F)\leq \lambda W_{\rm hom}(A)+(1-\lambda)W_{\rm hom}(B)+C_\delta,
$$
where $C_\delta\to0$ as $\delta\to0$. Hence, the claim on rank-one convexity follows.

Finally, the continuity of $W_{\rm hom}$ on $\Sigma$ can be deduced verbatim as in \cite[Step 2 of the proof of Theorem 3.1]{Conti08} by showing that $W_{\rm hom}$ is separately convex in suitable local variables.
\end{proof}

In the next lemma, we prove that the examples mentioned in Remark~\ref{r.onAss} indeed satisfy Assumption~\ref{ass}.
\begin{lemma}\label{L:pqexample}
Let $W:\R^d\times \Sigma\to[0,+\infty)$ be a Carath\'eodory-function such that $W(\cdot,F)$ is $Y$-periodic for every $F\in\Sigma$. Moreover, assume that $W$ satisfies one of the following two conditions:
\begin{enumerate}[(I)]
	\item (standard $p$-growth): There exist $p\geq 1$ and  $c\geq 1$ such that for a.e. $x\in\R^d$ and all $F\in\Sigma$
\begin{align}
\frac{1}{c}|F|^p-c&\leq W(x,F)\leq c(|F|^p+1),\label{ineq:pgrowth:ass0}
\end{align}
	\item (additional dependence on cofactors): There exist $p,q\geq 1$ and  $c\geq 1$ such that for a.e. $x\in\R^d$ and all $F\in\Sigma$
	\begin{align}
\frac{1}{c}\left(|F|^p+|{\rm adj}F|^q\right)-c\leq W(x,F)\leq c\left(|F|^p+|{\rm adj}F|^q\right)+c,\label{ineq:pgrowth:ass1}
\end{align}
	where ${\rm adj}F$ denotes the adjunct matrix of $F$.
	\end{enumerate}
Then $W$ satisfies Assumption~\ref{ass}.
\end{lemma}
\begin{proof}
The two-sided growth conditions \eqref{ineq:pgrowth:ass0} and \eqref{ineq:pgrowth:ass1}, together with multiplicativity property of the adjunct matrix (see e.g.~\cite[Proposition 5.66 (i)]{Dac}), imply the submultiplicative upper bound \eqref{eq:submultiplicative} and the lower bound on $W(x,F)$ of \eqref{eq:W_hom_comparison} with a constant depending on $c$ in \eqref{ineq:pgrowth:ass0} and \eqref{ineq:pgrowth:ass1} and in addition on $p$. The only slightly non-trivial part is to verify the upper bound in \eqref{eq:W_hom_comparison}. This follows by the same argument as for the lower bound on $W_{\rm hom}$ in Lemma~\ref{L:generalcellformula}, see \eqref{est:keylbwhom}. Indeed, in the case (I), by Jensen's inequality we have for all $\varphi\in W_0^{1,\infty}(kY)$
\begin{equation}\label{pf:pgrowthwhom}
	c^{-1}|F|^p-c\leq c^{-1}\left|\fint_{kY}F+\nabla\varphi\dy\right|^p-c\leq \fint_{kY}c^{-1}|F+\nabla\varphi|^p-c\dy\overset{\eqref{ineq:pgrowth:ass0}}{\leq} \fint_{kY}W(y,F+\nabla\varphi)\dy.
\end{equation}
Hence, we deduce as in Lemma~\ref{L:generalcellformula} that $c^{-1}|F|^p-c\leq W_{\rm hom}(F)$ and thus
$$
W_{\rm hom}(F)\geq c^{-1}|F|^p-c\stackrel{\eqref{ineq:pgrowth:ass0}}\geq c^{-1}(c^{-1} W(x,F)-1)-c
$$
which yields the upper bound in \eqref{eq:W_hom_comparison}. The above argument can be extended to the case (II) appealing to the fact that $F \mapsto {\rm  adj}F$ is a null Lagrangian (see \cite[Lemma 5.5 (ii)]{Dac}.

%
%
%
%
%
%
\end{proof}
Finally, we recall a version of the well-known Scorza-Dragoni theorem \cite[Theorem 6.35]{FoLe}. This will be important in the next section in order to deal with the fact that $W$ is only assumed to measurable in the spatial variable.
\begin{theorem}\label{T:scorza}
Let $E\subset\R^d$ be a Lebesgue-measurable set and $B\subset\R^m$ be a Borel set, and let $f:E\times B\to \R$ be a Carath\'eodory function. Then for every $\sigma>0$ there exists a closed set $K_{\sigma}\subset E$, with $|E\setminus K_{\sigma}|\leq\sigma$, such that $f$ restricted to $K_{\sigma}\times B$ is continuous.
\end{theorem}

We will use the following consequence of Theorem~\ref{T:scorza}.
\begin{corollary}\label{c.uniformcont}
Suppose that $W$ satisfies Assumption~\ref{ass}. For every $R>0$ and $\sigma,\eta\in(0,1]$ there exists $\rho\in(0,1]$ and a compact set $K_\sigma\subset Y$ satisfying $|Y\setminus K_\sigma|\leq \sigma$ such that
\begin{equation}
\xi,\zeta\in\Sigma \quad\wedge\quad |\xi|,\,|\zeta|\leq R\quad\wedge\quad |\xi-\zeta|\leq \rho \qquad\Rightarrow\qquad \sup_{Y\in \Z^d+K_\sigma}|W(y,\xi)-W(y,\zeta)|\leq \eta.
\end{equation}
\end{corollary}

\begin{proof}
By Theorem~\ref{T:scorza} with $E=Y$ and $B=\Sigma$ there exists for every $\sigma\in(0,1]$ a closed set $K_\sigma\subset Y$ with $|Y\setminus K_\sigma|\leq \sigma$, such that $W$ restricted to $K_\sigma \times \Sigma$ is continuous and thus uniformly continuous on $K_\sigma\times (\Sigma \cap \{F \in \R^{d\times d}\,:\,|F|\leq R\})$. From this and the periodicity of $W$ the claim follows. 
\end{proof}

%
\section{Construction of a recovery sequence - proof of Theorem~\ref{thm:upperbound}}\label{sec:proofs}
In the next lemma we extend the local construction of \cite[Lemma 4.2]{CoDo} from the setting of relaxation to the one of homogenization that involves the additional micro-scale $\e$.
\begin{lemma}\label{l.localconstruction}
For all $F\in\R^{d\times d}$ and $\eta\in(0,1]$ there exists $\delta>0$ such that the following is true: for every ball $B=B(x_0,r)\subset\R^d$ and for every $u\in W^{1,1}(B)^d$ satisfying
\begin{equation}\label{L:locrec:asssmall}
\fint_B|\nabla u-F|+|W_{\rm hom}(\nabla u)-W_{\rm hom}(F)|\dx\leq \delta,
\end{equation}
there exists a sequence $(z_\e)_\e\subset W^{1,1}(B)^d$ with $z_\e=u$ on $\partial B$ and
\begin{equation}\label{L:conti42:claim1}
\limsup_{\e\to0}\fint_B W(\tfrac{x}\e,\nabla z_\e)\dx\leq \fint_B W_{\rm hom}(\nabla u)+\eta\dx.
\end{equation}
Additionally, there exists $C<+\infty$ depending on the dimension $d$ and the constant $c\in[1,+\infty)$ in Assumption \ref{ass} such that
\begin{equation}\label{L:conti42:claim2}
\limsup_{\e\to0}\fint_B |u-z_\e|\dx\leq Cr \fint_B 1+W_{\rm hom}(\nabla u)\dx.
\end{equation}

\end{lemma}

\begin{proof}
In the proof many quantities depend on $F$ and $\eta$, but in order to reduce notation we only indicate the dependence on $\eta$, implicitly allowing for a dependence on $F$, too. By definition of $W_{\rm hom}$, we find $k_{\eta}\in\N$ and $\varphi_\eta\in W_0^{1,\infty}(k_{\eta}Y)^d$ such that
\begin{equation}\label{pf:Lconti42:varphieta}
\fint_{k_{\eta}Y}W(y,F+\nabla \varphi_\eta(y))\dy\leq W_{\rm hom}(F)+\eta.
\end{equation}
We extend $\varphi_\eta\in W_0^{1,\infty}(k_{\eta}Y)^d$ $k_{\eta}Y$-periodically and identify $\varphi_\eta$ with this extension. Moreover, we define $\varphi_{\eta,\e}:=\e\varphi_\eta(\frac{\cdot}\e)$ and set
\begin{equation}
\phi_{\eta,\e}(x):=\begin{cases}Fx+\varphi_{\eta,\e}(x)&\mbox{if $x\in B^{(\e)}$,}\\ Fx&\mbox{otherwise,}\end{cases}
\end{equation}
where 
\begin{equation}\label{eq:defB^eps}
B^{(\e)}:=\bigcup_{\substack{q\in\e k_{\eta} \Z^d\\ q+\e k_{\eta}Y\subset B}} (q+\e k_{\eta} Y).
\end{equation}
By construction we then have that
\begin{equation}\label{pf:L42conti:propphietae}
(i)\;\phi_{\eta,\e}\in W^{1,\infty}(B)^d,\quad (ii)\;\phi_{\eta,\e}(x)=Fx\mbox{ for all $x\in B\setminus B^{(\e)}$}\quad (iii)\;\det \nabla \phi_{\eta,\e}=1\mbox{ a.e. in $B$,}
\end{equation}
where (iii) follows from the fact that in view of \eqref{pf:Lconti42:varphieta} we have  $\det(\nabla \phi_{\eta,\e})=\det(F+\nabla \varphi_\eta(\tfrac{x}{\e}))=1$ a.e.\ on $B^{(\e)}$ and  $\det(\nabla \phi_{\eta,\e})=\det(F)=1$ on $B\setminus B^{(\e)}$. Next, we define
\begin{equation}
v_\e=F^{-1}\phi_{\eta,\e}.
\end{equation}
We argue that $v_{\e}$ is a bi-Lipschitz map. To this end, note that since $\det(\nabla v_{\e})=1$ a.e., for some (actually any) $q>d$ we have that
\begin{equation*}
	\int_B|(\nabla v_{\e}(x))^{-1}|^q\det(\nabla v_{\e}(x))\dx=\int_B |{\rm adj}(\nabla v_{\e}(x))|^q\dx<+\infty.
\end{equation*}
Since further $v_{\e}(x)=x$ on $\partial B$, a direct application of \cite[Theorem 2]{Ball81} yields that $v_{\e}$ is a bi-Lipschitz map from $B$ to itself. In conclusion, we can define
\begin{equation}
z_\e:=u\circ v_\e\in W^{1,1}(B)^d,
\end{equation}
which satisfies $z_\e=u$ on $\partial B$ in the sense of traces. In order to show \eqref{L:conti42:claim1}, we decompose
\begin{equation}
\fint_B W(\tfrac{x}\e,\nabla z_\e)\dx- \fint_B W_{\rm hom}(\nabla u)\dx=(I)_\e+(II)_\e+(III),
\end{equation}
where
\begin{align*}
(I)_\e:=&\fint_B W(\tfrac{x}\e,\nabla z_\e)-W(\tfrac{x}\e,F+\nabla \varphi_{\eta,\e})\dx,\quad
(II)_\e:=&\fint_B W(\tfrac{x}\e,F+\nabla \varphi_{\eta}(\tfrac{x}\e))-W_{\rm hom}(F)\dx,\\
(III):=&\fint_B W_{\rm hom}(F)-W_{\rm hom}(\nabla u)\dx.
\end{align*}
We consider these terms separately. By the $k_{\eta}Y$-periodicity of $W(\cdot,F+\nabla \varphi_{\eta})$, we have that
\begin{equation}\label{lim:IIeps}
\lim_{\e\to0} (II)_\e=\fint_{k_{\eta}Y} W(y,F+\nabla \varphi_{\eta})\dy-W_{\rm hom}(F)\stackrel{\eqref{pf:Lconti42:varphieta}}\leq \eta.
\end{equation}
In order to estimate $(III)$ and $(I)_\e$, we will appeal to the continuity of $\xi\mapsto W(y,\xi)$ (in the sense of Corollary~\ref{c.uniformcont}) and of $W_{\rm hom}$. Note that
\begin{equation}\label{def:Reta}
\|\nabla v_\e\|_{L^\infty(B)}= \|F^{-1}\nabla \phi_{\eta,\e}\|_{L^\infty(B)}\leq |{\rm Id}|+\|F^{-1}\nabla \varphi_\eta\|_{L^\infty(k_{\eta}Y)}=:R_\eta<+\infty
\end{equation}
and
\begin{equation}\label{eq:defMeta}
\|\nabla \phi_{\eta,\e}\|_{L^\infty(B)}\leq  |F|+\|\nabla \varphi_\eta\|_{L^\infty(k_{\eta}Y)}=:M_\eta<+\infty.
\end{equation}
Let
\begin{equation}\label{eq:equiintegrable}
	\sigma_{\eta}=\frac{\eta}{4c(\sup_{\xi\in\Sigma,\,|\xi|\leq R_{\eta}}W_{\rm hom}(\xi)+1)(W_{\rm hom}(F)+1)}.
\end{equation}
Appealing to the continuity of $W_{\rm hom}$ (see Lemma~\ref{L:Whom:cont}) and Corollary~\ref{c.uniformcont}, we find a compact set $C_{\eta}\subset Y$ with $|Y\setminus C_{\eta}|\leq\sigma_{\eta}$ and $\rho_{\eta}\in(0,1)$ sufficiently small such that, for all $\xi,\zeta\in\Sigma$,
\begin{equation}\label{eq:Whomcontinuity}
|\xi-F|\leq \rho_{\eta}\quad\Rightarrow\quad |W_{\rm hom}(\xi)-W_{\rm hom}(F)|\leq \eta
\end{equation}
and
\begin{equation}\label{eq:Wconinuity}
|\xi-\zeta|\leq \rho_{\eta} R_\eta \quad \wedge\quad |\xi|\leq M_\eta  \qquad\Rightarrow\qquad \|W(\tfrac{\cdot}{\e},\xi)-W(\tfrac{\cdot}{\e},\zeta)\|_{L^{\infty}(\e(\Z^d+ C_{\eta}))}\leq \eta.
\end{equation}

Hence, due to the non-negativity of $W_{\rm hom}(\nabla u)$,
\begin{align}\label{est:IIIeps}
(III)=&\frac1{|B|}\int_{B\cap \{|\nabla u-F|\leq \rho_{\eta}\}}W_{\rm hom}(F)-W_{\rm hom}(\nabla u)\dx+\frac1{|B|}\int_{B\cap \{|\nabla u-F|> \rho_{\eta}\}}W_{\rm hom}(F)-W_{\rm hom}(\nabla u)\dx\notag\\
\overset{\eqref{eq:Whomcontinuity}}{\leq}&\eta+\frac{W_{\rm hom}(F)}{\rho_{\eta}}\fint_{B}|\nabla u-F|\dx\notag\\
\leq &\eta +\frac{W_{\rm hom}(F)\delta}{\rho_{\eta}}.
\end{align}
To estimate $(I)_\e$, we decompose $B$ as $B=B^{(\e)}\cup (B\setminus B^{(\e)})$ with $B^{(\e)}$ defined in \eqref{eq:defB^eps}. On $B\setminus B^{(\e)}$ it holds that $z_{\e}=u$, so that by the non-negativity of $W$ we have
\begin{align*}
\limsup_{\e\to0}\int_{B\setminus B^{(\e)}}W(\tfrac{x}\e,\nabla z_\e)-W(\tfrac{x}\e,F+\nabla \varphi_{\eta}(\tfrac{x}\e))\dx&\leq\limsup_{\e\to 0} \int_{B\setminus B^{(\e)}}W(\tfrac{x}{\e},\nabla u)\dx
\\
&\overset{\eqref{eq:W_hom_comparison}}{\leq} \limsup_{\e\to 0} c\int_{B\setminus B^{(\e)}}W_{\rm hom}(\nabla u)+1\dx=0,
\end{align*}
where we used that $|B\setminus B^{(\e)}|\to 0$ as $\e\to 0$ for the last equality.
On $B^{(\e)}$, we have that
\begin{equation}\label{eq:nablazeps}
\nabla z_\e=((\nabla u)\circ v_\e)\nabla v_\e=((\nabla u-F)\circ v_\e)\nabla v_\e+F+\nabla \varphi_\eta(\tfrac{\cdot}\e).
\end{equation}
Set $\omega_\e:=\{x\in B^{(\e)}\,|\, |(\nabla u-F)\circ v_\e|>\rho_{\eta}\}$. Equation \eqref{eq:nablazeps} implies that on $B\setminus \omega_\e$
\begin{equation}\label{eq:nablazeps1}
|\nabla z_\e-(F+\nabla \varphi_\eta(\tfrac{\cdot}\e))|\leq \|\nabla v_\e\|_{L^\infty(B)}|(\nabla u-F)\circ v_\e)|\stackrel{\eqref{def:Reta}}\leq \rho_\eta R_\eta
\end{equation}
and thus the continuity property \eqref{eq:Wconinuity} and the definition of $M_{\eta}$ in \eqref{eq:defMeta} yield that
\begin{equation}
\|W(\tfrac{\cdot}{\e},\nabla z_\e)-W(\tfrac{\cdot}{\e},F+\nabla \varphi_\eta(\tfrac{\cdot}\e))\|_{L^\infty(B\setminus \omega_\e\cap(\e(\Z^d+ C_{\eta})))}\leq \eta
\end{equation}
which yields
\begin{equation}\label{est:finalIepsa}
\frac{1}{|B|}\int_{B\setminus\omega_\e\cap(\e(\Z^d+C_{\eta}))}W(\tfrac{x}\e,\nabla z_\e)-W(\tfrac{x}\e,F+\nabla \varphi_{\eta,\e})\dx\leq  \eta.
\end{equation}

To bound the contribution from the set $B\setminus \omega_{\e}\cap(\e(\Z^d+Y\setminus C_{\eta}))$, we use \eqref{eq:nablazeps}, the non-negativity of $W$ and its submultiplicative bound, see \eqref{eq:submultiplicative}, to deduce that
\begin{eqnarray}\label{eq:applyarea}
	& &\quad \frac{1}{|B|}\int_{B\setminus \omega_{\e}\cap(\e(\Z^d+Y\setminus C_{\eta}))}W(\tfrac{x}{\e},\nabla z_{\e})-W(\tfrac{x}{\e},F+\nabla\varphi_{\eta,\e})\dx\nonumber
	\\
	&\stackrel{\eqref{eq:submultiplicative}}\leq& \frac{c}{|B|}\int_{B\cap(\e(\Z^d+Y\setminus C_{\eta}))}(W(\tfrac{x}{\e},\nabla v_{\e})+1)(W(\tfrac{x}{\e},\nabla u\circ v_{\e})+1)\dx\nonumber
	\\
	&\stackrel{\eqref{eq:W_hom_comparison}}\leq&  \frac{4c^2}{|B|}\int_{B\cap(\e(\Z^d+Y\setminus C_{\eta}))}(W_{\rm hom}(\nabla v_{\e})+1)(W_{\rm hom}(\nabla u\circ v_{\e})+1)\dx\nonumber
	\\
	&\stackrel{\eqref{def:Reta}}\leq& \underbrace{\frac{4c^2}{|B|}\sup_{\xi\in\Sigma,\,|\xi|\leq R_{\eta}}(W_{\rm hom}(\xi)+1)}_{=:\tfrac{c_{\eta}}{|B|}}\int_{B\cap(\e(\Z^d+Y\setminus C_{\eta}))}(W_{\rm hom}(\nabla u\circ v_{\e})+1)\dx\nonumber
	\\
	&=&\frac{c_\eta}{|B|}\int_{v_{\e}(B\cap (\e(\Z^d+Y\setminus C_{\eta})))}(W_{\rm hom}(\nabla u)+1)\dx\nonumber 
	\\
	&\leq& c_{\eta}\fint_B |W_{\rm hom}(\nabla u)-W_{\rm hom}(F)|\dx+\frac{c_{\eta}}{|B|}\int_{v_{\e}(B\cap (\e(\Z^d+Y\setminus C_{\eta})))}W_{\rm hom}(F)+1\dx\nonumber
	\\
	&\leq& c_{\eta}\delta+\frac{c_{\eta}}{|B|}|v_{\e}(B\cap (\e(\Z^d+Y\setminus C_{\eta})))|(W_{\rm hom}(F)+1),
\end{eqnarray}
where we used the area formula for the bi-Lipschitz transformation $v_{\e}:B\to B$ that satisfies $\det(\nabla v_{\e})=1$ from the fourth to the fifth line and $v_\e(B)=B$ from the fifth to the sixth line. Again by the area formula, the Riemann-Lebesgue Lemma and the bound $|Y\setminus C_\eta|\leq \sigma_\eta$, we have that
\begin{equation*}
	|v_{\e}(B\cap (\e(\Z^d+Y\setminus C_{\eta})))|=|(B\cap (\e(\Z^d+Y\setminus C_{\eta})))|\overset{\e\to 0}{\rightarrow} |Y\setminus C_{\eta}||B|\leq \sigma_{\eta}|B|.
\end{equation*}
Hence, from the choice of $\sigma_{\eta}$ (cf. \eqref{eq:equiintegrable}) and \eqref{eq:applyarea} we infer that
\begin{equation}\label{eq:finalIepsc}
	\limsup_{\e\to 0}\frac{1}{|B|}\int_{B\setminus \omega_{\e}\cap(\e(\Z^d+Y\setminus C_{\sigma}))}W(\tfrac{x}{\e},\nabla z_{\e})-W(\tfrac{x}{\e},F+\nabla\varphi_{\eta,\e})\dx\leq c_{\eta}\delta+\eta.
\end{equation}

On the set $\omega_{\e}$, the inequality $|(\nabla u-F)\circ v_{\e}|\geq \rho_{\eta}$ implies that
\begin{align*}
W_{\rm hom}(\nabla u\circ v_{\e})+1&\leq |W_{\rm hom}(\nabla u\circ v_{\e})-W_{\rm hom}(F)|+W_{\rm hom}(F)+1
\\
&\leq |W_{\rm hom}(\nabla u\circ v_{\e})-W_{\rm hom}(F)|+\frac{W_{\rm hom}(F)+1}{\rho_\eta}|(\nabla u-F)\circ v_{\e}|
\end{align*}
and thus, inserting the upper bounds on $W$, we obtain with analogous estimates as in (the the first four lines of) \eqref{eq:applyarea} that
\begin{align}
\int_{\omega_\e}W(\tfrac{x}\e,\nabla z_\e)\dx\leq& c_\eta\int_{\omega_\e}(W_{\rm hom}(\nabla u\circ v_{\e})+1)\dx\notag\\
\leq& \underbrace{c_\eta\biggl(1+\frac{W_{\rm hom}(F)+1}{\rho_{\eta}}\biggr)}_{=:\widetilde{c}_{\eta}}\int_{\omega_\e}|W_{\rm hom}(\nabla u\circ v_{\e})-W_{\rm hom}(F)|+|(\nabla u-F)\circ v_\e|\dx,
\end{align}
where $c_\eta$ is defined as in \eqref{eq:applyarea}. Since $v_\e$ is a bi-Lipschitz map from $B$ onto $B$ with $\det(\nabla v_\e)=1$ a.e., again by the area-formula (using $\omega_\e\subset B$ and $W\geq0$) 
\begin{align}\label{est:finalIepsb}
\int_{\omega_\e}W(\tfrac{x}\e,\nabla z_\e)\dx&\leq \widetilde{c}_{\eta}\int_{B}|W_{\rm hom}(\nabla u\circ v_{\e})-W_{\rm hom}(F)|+|(\nabla u-F)\circ v_\e|\dx\nonumber
\\
&= \widetilde{c}_{\eta}\int_{B}|W_{\rm hom}(\nabla u)-W_{\rm hom}(F)|+|\nabla u-F|\dx\stackrel{\eqref{L:locrec:asssmall}}\leq \widetilde{c}_{\eta}\delta |B|.
\end{align}
Hence, combining \eqref{lim:IIeps}, \eqref{est:IIIeps} \eqref{est:finalIepsa}, \eqref{eq:finalIepsc} and \eqref{est:finalIepsb} we obtain
\begin{align*}
\limsup_{\e\to0}\fint_B W(\tfrac{x}\e,\nabla z_\e)\dx- \fint_B W_{\rm hom}(\nabla u)\dx\leq 4\eta+\frac{W_{\rm hom}(F)\delta}{\rho_\eta}+(c_{\eta}+\widetilde{c}_{\eta})\delta.
\end{align*}
Choosing $\delta>0$ sufficiently small such that $(c_{\eta}+\widetilde{c}_{\eta})\delta+\frac{W_{\rm hom}(F)\delta}{\rho_{\eta}}\leq \eta$, we obtain \eqref{L:conti42:claim1} with $\eta$ replaced by $5\eta$ and the claim follows by redefining $\eta$.

Finally, we show \eqref{L:conti42:claim2}. This follows from Poincar\'e's inequality (using $u-z_\e=0$ on $\partial B$) and the lower bound in \eqref{eq:W_hom_comparison} in the form
\begin{align*}
\fint_{B}|u-z_\e|\dx\leq& Cr\fint_{B}|\nabla (u-z_\e)|
\leq Cr\fint_{B}|\nabla u|+|\nabla z_\e|\dx\\
\leq& Cr\fint_{B}|\nabla u|+cW(\tfrac{x}\e,\nabla z_\e)+c^2\dx
\end{align*}
and thus \eqref{L:conti42:claim2} follows with help of \eqref{L:conti42:claim1} and \eqref{ass:Whomgrowth}.
\end{proof}
With the above local construction, similar to \cite{CoDo} we obtain the global 'recovery' sequence via a covering argument.
\begin{lemma}\label{l.globalconstruction}
Let $u\in W^{1,1}(\Omega)^d$. Then there exists a sequence $u_{\e}\in u+W^{1,1}_0(\Omega)^d$ such that $u_{\e}\to u$ in $L^1(\Omega)^d$ satisfying also
\begin{equation}\label{eq:recovery}
	\limsup_{\e\to 0}\int_{\Omega}W(\tfrac{x}{\e},\nabla u_{\e})\,\dx\leq \int_{\Omega}W_{\rm hom}(\nabla u)\dx.
\end{equation}
If $u\in W^{1,\infty}(\Omega)^d$, then we can assume that $u_{\e}\in W^{1,\infty}(\Omega)^d$, too. 
\end{lemma}
\begin{proof}
We can assume that $W_{\rm hom}(\nabla u)\in L^1(\Omega)$. Otherwise we take $u_{\e}=u$.

\step 1 Fix $\eta\in(0,1]$. We claim that for all $j\in\N$ there exist a sequence $(u_\e^j)_\e\subset u+W_0^{1,1}(\Omega)^d$ and an open set $\Omega_j\subset \Omega$ satisfying
\begin{align}
	|\Omega_j|\leq& 2^{-j}|\Omega|\quad\mbox{and}\quad u_\e^j\equiv u\quad\mbox{on $\Omega_j$},\label{L:rec:s1:ind3}\\
	\limsup_{\e\to0}\int_{\Omega\setminus \Omega_j}W(\tfrac{x}{\e},\nabla u_{\e}^j)\dx \leq&\int_{\Omega\setminus \Omega_j}W_{\rm hom}(\nabla u)+\eta\dx,\label{L:rec:s1:ind1}\\
	\limsup_{\e\to0}\int_{\Omega\setminus \Omega_j}|u_{\e}^j-u|\dx\leq& C\eta\int_{\Omega\setminus \Omega_j}1+W_{\rm hom}(\nabla u)\dx,\label{L:rec:s1:ind2}
\end{align}
where $C=C(c,d)\in[1,+\infty)$ is the constant in \eqref{L:conti42:claim2} (in particular it is independent of $\eta$).

We prove the claim by induction. For $j=0$, we set $\Omega_0=\Omega$ and $u_\e^0=u$ and \eqref{L:rec:s1:ind3}--\eqref{L:rec:s1:ind2} hold trivially.

Fix $j\in \N$ and suppose that $(u_\e^j)_\e\subset u+W_0^{1,1}(\Omega)$ and the open set $\Omega_j\subset \Omega$ satisfy \eqref{L:rec:s1:ind3}--\eqref{L:rec:s1:ind2}. Let $L\subset \Omega$ be the set of Lebesgue points of $\nabla u$ and $W_{\rm hom}(\nabla u)$ and for $x\in L$ set $F(x)=\nabla u(x)$. Let $\delta(x)>0$ be given by Lemma \ref{l.localconstruction} for the choice $F=F(x)$ and $\eta>0$. In order to define $u_{\e}^{j+1}$ and the open set $\Omega_{j+1}\subset\Omega$, we choose for any $x\in L\cap\Omega_j$ a radius $r_j(x)\in (0,\eta)$ such that $B_{r_j(x)}(x)\subset\Omega_j$ and
\begin{align*}
&\quad\fint_{B_r(x)}|\nabla u_{\e}^j-F(x)|+|W_{\rm hom}(\nabla u_{\e}^j)-W_{\rm hom}(F(x))|\dy 
\\
&=\fint_{B_r(x)}|\nabla u-F(x)|+|W_{\rm hom}(\nabla u)-W_{\rm hom}(F(x))|\dy\leq \delta(x)
\end{align*}
for all $0<r<r_j(x)$. By the Vitali-Besicovitch covering theorem (\cite[Theorem 1.150 and Remark 1.151]{FoLe}) there exists a countable family of disjoint balls $\{B_{r_k}(x_k)\}_{k\in\N}$ that covers $\Omega_j$ up to a null set. We choose finitely many balls $\{B_{r_k}(x_k)\}_{k=1}^{N_j}$ such that $|\cup_{k=1}^{N_j}B_{r_k}(x_k)|\geq \frac12|\Omega_j|$. In each of these balls we apply Lemma \ref{l.localconstruction} to $F=F(x)$, the given $\eta>0$ and the function $u\in W^{1,1}(B_{r_k}(x_k))^d$ to obtain the corresponding family $z_{\e}^k\in u+W^{1,1}_0(B_{r_k}(x_k))^d$ such that
\begin{align}
	&\limsup_{\e\to 0}\int_{B_{r_k}(x_k)}W(\tfrac{x}{\e},\nabla z_\e^k)\dx\leq \int_{B_{r_k}(x_k)}W_{\rm hom}(\nabla u)+\eta\dx,\label{eq:energyest}
	\\
&\limsup_{\e\to 0}\int_{B_{r_k}(x_k)}|u-z_{\e}^k|\leq Cr_k\int_{B_{r_k}(x_k)}1+W_{\rm hom}(\nabla u)\dx.\label{eq:Lpest}
\end{align} 
Let us define
\begin{equation*}
	u_{\e}^{j+1}:=\begin{cases}
		 u_{\e}^j &\mbox{on $\Omega\setminus \bigcup_{k=1}^{N_j}B_{r_k}(x_k)$,}
		 \\
		 z_{\e}^k &\mbox{on $B_{r_k}(x_k),\, 1\leq k\leq N_j$,}
	\end{cases}\qquad\mbox{and}\qquad\Omega_{j+1}:=\Omega_j\setminus \bigcup_{k=1}^{N_j}\overline{B_{r_k}(x_k)},
\end{equation*}
so that $\Omega_{j+1}$ is open, $u_{\e}^{j+1}=u_{\e}^j=u$ on $\Omega_{j+1}$ and, since $B_{r_k}(x_k)\subset\Omega_j$, the property $u_{\e}^j=u$ on $\Omega_j$ and $z_{\e}^k\in u+W_0^{1,1}(B_{r_k}(x_k))$ imply that $u_{\e}^{j+1}\in u+ W^{1,1}_0(\Omega)^d$. Moreover, we have 
$$
|\Omega_{j+1}|\leq |\Omega_j|-|\bigcup_{k=1}^{N_j}B_{r_k}(x_k)|\leq\frac12|\Omega_j|\leq2^{-(j+1)}|\Omega|
$$
and thus \eqref{L:rec:s1:ind3} is valid for $j+1$. Finally, we show the estimates \eqref{L:rec:s1:ind1} and \eqref{L:rec:s1:ind2} with $j$ replaced by $j+1$. By construction
$$\Omega\setminus \Omega_{j+1}=(\Omega\setminus \Omega_{j})\cup\bigcup_{k=1}^{N_j}\overline{B_{r_k}(x_k)}$$
and thus by \eqref{L:rec:s1:ind1} and \eqref{eq:energyest}
\begin{align*}
	\limsup_{\e\to0}\int_{\Omega\setminus \Omega_{j+1}}W(\tfrac{x}{\e},\nabla u_{\e}^{j+1})\dx &\leq\limsup_{\e\to0}\int_{\Omega\setminus \Omega_j}W(\tfrac{x}{\e},\nabla u_{\e}^j)\dx +\sum_{k=1}^{N_j}\limsup_{\e\to0}\int_{B_{r_k}(x_k)}W(\tfrac{x}{\e},\nabla z_{\e}^k)\dx\\
	&\leq  \int_{\Omega\setminus\Omega_j}W_{\rm hom}(\nabla u)+\eta\dx+\sum_{k=1}^{N_j}\int_{B_{r_k}(x_k)}W_{\rm hom}(\nabla u)+\eta\dx\\
	&=  \int_{\Omega\setminus\Omega_{j+1}}W_{\rm hom}(\nabla u)+\eta\dx
\end{align*}
which proves \eqref{L:rec:s1:ind1} for $j$ replaced by $j+1$. Similarly, exploiting \eqref{L:rec:s1:ind2} and \eqref{eq:Lpest} we have
\begin{align*}
	\limsup_{\e\to0}\int_{\Omega}|u_{\e}^{j+1}-u|\dx\leq&\limsup_{\e\to0}\int_{\Omega\setminus \Omega_j}|u_{\e}^j-u|\dx +\sum_{k=1}^{N_j}\limsup_{\e\to0}\int_{B_{r_k}(x_k)}| z_{\e}^k-u|\dx\\
	\leq& C\eta\int_{\Omega\setminus \Omega_j}1+W_{\rm hom}(\nabla u)\dx+\sum_{k=1}^{N_j}Cr_k\int_{B_{r_k}(x_k)}1+W_{\rm hom}(\nabla u)\dx\\
	\leq& C\eta\int_{\Omega\setminus \Omega_{j+1}}1+W_{\rm hom}(\nabla u)\dx,
\end{align*}
where we used $r_k(x)\in(0,\eta)$ in the last inequality. This concludes the inductive construction.

\step 2 Conclusion. Fix $\eta\in(0,1]$. Appealing to Step~1, we find for every $j\in\mathbb N$ a sequence $(u_\e^j)_\e\subset u+W_0^{1,1}(\Omega)^d$ and an open set $\Omega_j\subset \Omega$ satisfying \eqref{L:rec:s1:ind3},
\begin{align*}
	\limsup_{\e\to0}\int_{\Omega}W(\tfrac{x}{\e},\nabla u_{\e}^j)\dx \leq&\limsup_{\e\to0}\int_{\Omega\setminus \Omega_j}W(\tfrac{x}{\e},\nabla u_{\e}^j)\dx +\limsup_{\e\to0}\int_{\Omega_j}W(\tfrac{x}{\e},\nabla u)\dx	\\
	\leq&\int_{\Omega\setminus \Omega_j}W_{\rm hom}(\nabla u)+\eta\dx +c\int_{\Omega_j}W_{\rm hom}(\nabla u)+1\dx
\end{align*}
and 
\begin{align*}
	\limsup_{\e\to0}\int_{\Omega}|u_{\e}^j-u|\dx\leq& C\eta\int_{\Omega}1+W_{\rm hom}(\nabla u)\dx.
\end{align*}
The above two limits in combination with $W_{\rm hom}(\nabla u)\in L^1(\Omega)$ and \eqref{L:rec:s1:ind3} imply 
\begin{align*}
	\limsup_{j\to+\infty}\limsup_{\e\to0}\int_{\Omega}W(\tfrac{x}{\e},\nabla u_{\e}^j)\dx \leq&\int_{\Omega}W_{\rm hom}(\nabla u)+\eta\dx, \\
	\limsup_{j\to+\infty}\limsup_{\e\to0}\int_{\Omega}|u_{\e}^j-u|\dx\leq& C\eta\int_{\Omega}1+W_{\rm hom}(\nabla u)\dx,
\end{align*}
and the claim follows from the arbitrariness of $\eta\in(0,1]$ and standard diagonal sequence arguments. Finally, if $u\in W^{1,\infty}(\Omega)^d$, then $u_{\e}\in W^{1,\infty}(\Omega)^d$ by the construction via composition in the proof of Lemma \ref{l.localconstruction}.
\end{proof}
\begin{corollary}\label{c.quasiconvex}
The function $W_{\rm hom}:\R^{d\times d}\to [0,+\infty]$ is quasiconvex in the sense of \eqref{eq:defquasiconvex}.
\end{corollary}
\begin{proof}
Let $F\in\R^{d\times d}$ and $\varphi\in W^{1,\infty}(B_1)^d$ such that $\varphi(x)=Fx$ on $\partial B_1$. We show that
\begin{equation}\label{ineq:key:whomquasiconvex}
W_{\rm hom}(F)\leq\fint_{B_1}W_{\rm hom}(\nabla\varphi)\dx.
\end{equation}

If $\det(F)\neq 1$, then one can argue as in the proof of Lemma \ref{L:generalcellformula} that
\begin{equation*}
	\fint_{B_1}W_{\rm hom}(\nabla\varphi)\dx=+\infty=W_{\rm hom}(F).
\end{equation*}
Let $\det(F)=1$. By Lemma \ref{l.globalconstruction} we find a sequence $u_{\e}\in W^{1,\infty}(B_1)^d$ such that $u_{\e}(x)=Fx$ on $\partial B_1$ and
\begin{equation*}
	\limsup_{\e\to 0}\fint_{B_1}W(\tfrac{x}{\e},\nabla u_{\e})\dx\leq \fint_{B_1}W_{\rm hom}(\nabla\varphi)\dx.
\end{equation*}
However, via a change of variables the left-hand side is bounded from below by $W_{\rm hom}(F)$ due to Lemma~\ref{L:generalcellformula} and the boundary conditions of $u_{\e}$. This yields the claimed inequality \eqref{ineq:key:whomquasiconvex}. The extension to any bounded, open set is standard.
\end{proof}

\begin{proof}[Proof of Theorem \ref{thm:upperbound}]
	It suffices to combine the Lemmata \ref{L:generalcellformula}, \ref{L:Whom:cont}, and \ref{l.globalconstruction} with Corollary \ref{c.quasiconvex}.
\end{proof}

\begin{remark}[Compressible materials]\label{r.compressible}
As remarked above, Lemma~\ref{l.localconstruction} is a 'homogenization version' of related relaxation results in \cite{CoDo}. As mentioned in the introduction, \cite{CoDo} applies also to compressible materials with a blow up behavior if $\det (F)\to0$. Let us briefly explain why we are not able to extend the methods of \cite{CoDo} to homogenization in the compressible case. In \cite{CoDo}, the authors assume growth conditions of the form $\frac1c |F|^p+\frac1c\theta(\det(F))-c\leq W(F)\leq c(1+|F|^p+\theta(\det(F)))$, where $\theta$ is a convex function which may blow up at zero and satisfies $\theta(ab)\leq (1+\theta(a))(1+\theta(b))$. Assume this growth condition for $W(x,\cdot)$ and let us take the same ansatz as in the proof of Lemma~\ref{l.localconstruction}. At some point, we need to ensure $\theta(\nabla z_\e)\in L^1$. However, by the multiplicative upper bound on $\theta$ and the lower bound on $W$,  $\theta(\nabla z_\e)$ can be bounded from above only by the product of two $L^1$-functions. The trick in \cite{CoDo} is to make a shift in the definition of $v_\e$, that is, consider $v_\e=F^{-1}\phi_{\eta,\e}(\cdot-a)+a$ Then there exists $a=a_{\eta,\e}$ such that $\theta(\nabla z_\e)$ is in $L^1$, but the shift leads to a change of variables and we would need that
$$
\fint_{k_{\eta}Y}W(y+\e^{-1}a_{\eta,\e},F+\nabla \varphi_\eta(y))\dy
$$
is close to $W_{\rm hom}(F)$ which is not true in general. Indeed, due to the periodicity of $W$ one can assume that $\e^{-1}a_{\eta,\e}$ converges to some element $a_0\in\overline{Y}$. With some effort one can show that the above integral then converges to the corresponding version with $\e^{-1}a_{\eta,\e}$ replaced by $a_0$, so that the oscillations of $\nabla\varphi_{\eta}$ are not 'almost optimal' in energy.
\end{remark}

\section{Remarks on the lower bound}\label{sec:LB}

In this final section, we present a lower bound on the $\Gamma$-$\liminf$ of the functional $F_\e$. In order to conserve the constraint of incompressibility, it is in general necessary to assume a stronger lower bound on the stored energy density of the form
\begin{equation}\label{p:growth:1}
\exists c\geq1,\,p\geq d:\quad c^{-1}|F|^p-c\leq W(x,F)\quad\mbox{for all $F\in\Sigma$ and a.e.\ $x\in\R^d$;}
\end{equation}
cf. \cite{KoRiWi}. Since $\Sigma\subset\R^{d\times d}$ is closed, there exists an extension of $W$ from $\R^d\times\Sigma$ to $\R^d\times\R^{d\times d}$ that is continuous in the second variable. We denote this extension by $\widetilde{W}$. Using explicit formulas for such an extension (cf. the proof of \cite[Theorem 7.2]{Deimling}) it follows that $\widetilde{W}$ can be taken as a Carath\'eodory-function $\widetilde{W}:\R^d\times\R^{d\times d}\to [0,+\infty)$ that is periodic in the first variable \footnote{Whenever $W(x,\cdot)$ is polyconvex for a.e. $x\in\R^d$, one can construct a finite, continuous extension that is also polyconvex and non-negative. Indeed, in this case $W(x,F)=w(x,(m(F),1))$ for some convex function $w(x,\cdot)$ defined on $\R^{k-1}\times\{1\}$ for some $k\in\N$ that counts the dimension of all minors and $m(F)$ containing all minors of $F$ of order $\leq d-1$ of $F$. Since $\R^{k-1}\times\{1\}$ is an affine subspace of $\R^k$, there exists an affine projection $p:\R^k\to\R^{k-1}\times\{1\}$. The formula $\widetilde{W}(x,F)=w(x,p(m(F),\det(F)))$ then provides a polyconvex extension of $W(x,\cdot)$ with the claimed properties. \label{footnote:extension}}. Upon replacing $\widetilde{W}$ by $\max\{\widetilde{W},\frac{1}{c}|F|^p-c\}$ we can assume that also $\widetilde{W}$ satisfies 
\begin{equation}\label{p:growth:tilde}
	c^{-1}|F|^p-c\leq \widetilde{W}(x,F)\quad\mbox{for all $F\in\R^{d\times d}$ and a.e.\ $x\in\R^d$,}
\end{equation}
We define $W_n:\R^d\times \R^{d\times d}\to[0,+\infty)$ by
\begin{equation}
	W_n(x,F)=\min\left\{\widetilde W(x,F),n(|F|^p+1)\right\}+n|\det(F)-1|.
\end{equation}
Since $p\geq d$ each $W_n$ satisfies $p$-growth assumptions of the form
\begin{equation*}
	\frac{1}{c}|F|^p-c\leq W_n(x,F)\leq c_n(|F|^p+1).
\end{equation*}
In particular, to each $W_n$ we can apply standard homogenization results (e.g. \cite{Mueller87}) to deduce that 
\begin{equation*}
	F_{\e,n}(u):=\int_{\Omega}W_n(\tfrac{x}{\e},\nabla u)\dx \overset{\Gamma}{\to}\int_{\Omega}\overline W_{n}(\nabla u)\dx
\end{equation*}
on $W^{1,p}(\Omega)^d$ (with respect to the $L^1$-topology) with
\begin{equation}\label{whomn}
	\overline W_{n}(F)=\inf_{k\in\N}\overline W_{n}^{(k)}(F)\quad\mbox{where}\quad \overline W_{n}^{(k)}(F):=\inf_{\varphi\in W^{1,p}_0(kY)^d}\fint_{kY}W_n(x,F+\nabla\varphi(x))\dx.
\end{equation}
Note further that $W_n\leq W_{n+1}$ for all $n\in\N$ and that $W_n\uparrow W$ as $n\to +\infty$. Let us define
\begin{equation}\label{eq:lowerbounddensity}
	\underbar{W}(F):=\sup_{n\in\N}\overline W_{n}(F).
\end{equation}
This integrand yields a lower bound for the $\Gamma\hbox{-}\liminf$ of $F_{\e}$:
\begin{proposition}\label{thm:lowerbound}
	Suppose the assumptions of Theorem~\ref{thm:upperbound} are satisfied. In addition assume \eqref{p:growth:1} and let $\underline{W}$ be defined as above. Then for every family $(u_{\e})_{\e>0}\subset W^{1,1}(\Omega)^d$ such that $u_{\e}\to u$ in $L^1(\Omega)^d$ it holds that
	\begin{equation*}
		\liminf_{\e\to 0}\int_{\Omega}W(\tfrac{x}{\e},\nabla u_{\e}(x))\dx\geq \int_{\Omega}\underbar{W}(\nabla u(x))\dx.
	\end{equation*}
	If the left-hand side is finite, then $u\in W^{1,p}(\Omega)^d$. The map $\underbar{W}$ is finite exactly on $\Sigma$, where it is continuous. Moreover, $\underbar{W}$ is quasiconvex in the sense of \eqref{eq:defquasiconvex}.\\
	Finally, assume in addition that $W(x,\cdot)$ is polyconvex for a.e.\ $x\in \R^d$. Then, for all $F\in\R^{d\times d}$,
	\begin{equation}\label{claim:lb:whomk}
		\forall k\in\N:\qquad \sup_{n\in\N} \overline W_{n}^{(k)}(F)=\inf_{\varphi\in W^{1,p}_0(kY)^d}\fint_{kY}W(x,F+\nabla\varphi(x))\dx=:\overline W^{(k)}(F).
	\end{equation}

\end{proposition}

\begin{proof}
	\step 1 We prove the liminf inequality and the properties of $\underline W$.
	
	Let $u_{\e}\to u$ in $L^1(\Omega)^d$. Then by the lower bound $W\geq W_n$ and the $\Gamma$-convergence result for integral functionals with standard $p$-growth we have that
	\begin{equation*}
		\liminf_{\e\to 0}\int_{\Omega}W(\tfrac{x}{\e},\nabla u_{\e})\dx\geq \liminf_{\e\to 0}\int_{\Omega}W_n(\tfrac{x}{\e},\nabla u_{\e})\dx\geq\int_{\Omega}\overline{W}_{n}(\nabla u)\dx.
	\end{equation*}
	Since $W_n\leq W_{n+1}$ also $\overline{W}_{n}\leq \overline{W}_{n+1}$ and therefore the lower bound of the energy follows by applying the monotone convergence theorem. The lower bound in \eqref{p:growth:tilde} and the same computations as in \eqref{pf:pgrowthwhom} imply the lower bound $\overline{W}_{n}(F)\geq \frac{1}{c}|F|^p-c$ and by monotonicity it also holds for $\underbar{W}$. Hence, if the left-hand side in the above estimate is finite, then $u\in W^{1,p}(\Omega)^d$. As $\underbar{W}\leq W_{\rm hom}$, the finiteness of the latter on $\Sigma$ ensures the same for $\underbar{W}$. If, however, $F\in\R^{d\times d}$ is such that $\det(F)\neq 1$, then
	\begin{align*}
		\underbar{W}(F)&\geq \overline{W}_{n}(F)=\inf_{k\in\N}\inf_{\varphi\in W^{1,p}_0(kY)^d}\fint_{kY}W_n(x,F+\nabla\varphi)\dx
		\\
		&\geq \inf_{k\in\N}\inf_{\varphi\in W^{1,p}_0(kY)^d}n\fint_{kY}|\det(F+\nabla\varphi)-1|\dx\overset{p\geq d}{\geq} n|\det(F)-1|.
	\end{align*}
	Letting $n\to +\infty$, it follows that $\underbar{W}(F)=+\infty$. Next, $\underbar{W}$ is quasiconvex as the supremum of quasiconvex functions (each $\overline{W}_{n}$ is quasiconvex by standard lower semicontinuity results for functionals with polynomial growth; cf. \cite[Theorem 8.1]{Dac}) and finally $\underbar{W}$ is continuous on $\Sigma$ due to \cite[Theorem 1.1]{Conti08}. 
	
	\step 2 We assume that $W(x,\cdot)$ is polyconvex for a.e.\ $x\in \R^d$ and prove \eqref{claim:lb:whomk}. As explained in Footnote \ref{footnote:extension}, in this case we can also assume that $\widetilde{W}(x,\cdot)$ is polyconvex. Fix $k\in\N$. Since $W_n\leq W$, we have $\overline W_n^{(k)}\leq \overline W^{(k)}$ for all $n\in\N$ and thus it suffices to show that $\sup_n \overline W_n^{(k)}\geq \overline W^{(k)}$. If $F\notin\Sigma$, then as in Step~1 one can show that $\sup_{n\in\N}\overline W_n^{(k)}(F)=+\infty$, so that it suffices to consider the case $F\in\Sigma$. For every $n\in\N$ let $\varphi_n\in W_0^{1,p}(kY)^d$ be such that
	$$
	\overline W_{n}^{(k)}(F)\geq \fint_{kY}W_n(x,F+\nabla  \varphi_n)\dx-\frac{1}{n}.
	$$ 
	Then \eqref{p:growth:1}, the almost minimality of $\varphi_n$, and the upper bound in \eqref{eq:W_hom_comparison} imply
	\begin{equation}\label{keyest:varphik:lb}
		\fint_{kY}\frac1c|F+\nabla \varphi_n|^p-c\dx\leq \fint_{kY}W_n(x,F+\nabla  \varphi_n)\dx\leq \fint_{kY}W_n(x,F)+\frac{1}{n}\dx\leq c(W_{\rm hom}(F)+1).
	\end{equation}
	Hence, there exist $\varphi\in W_0^{1,p}(Y)^d$ and a subsequence $(\varphi_{n_j})_j$ such that $\varphi_{n_j}\rightharpoonup \varphi$ in $W^{1,p}(kY)^d$. Next, we investigate the convergence of $\det(F+\nabla\varphi_n)$. From the definition of $W_n$ and almost minimality of $\varphi_n$ we infer that
	\begin{equation*}
		\fint_{kY}|\det(F+\nabla\varphi_n)-1|\dx\leq \frac{1}{n}\fint_{kY}W_n(x,F+\nabla\varphi_n)\dx\leq \frac{1}{n}\fint_{kY}W(x,F)+1\dx\overset{\eqref{eq:W_hom_comparison}}{\leq}\frac{c}{n}\left(W_{\rm hom}(F)+1\right),
	\end{equation*}
	so that $\det(F+\nabla\varphi_n)\to 1$ in $L^1(kY)$. At the same time, due to the divergence structure of the Jacobian, it is known that $\det(F+\nabla\varphi_{n_j})\rightharpoonup\det(F+\nabla\varphi)$ in $\mathcal D'(kY)$, see e.g.\ \cite[Theorem 8.20]{Dac}. Combining these two observations, we find that $\det(F+\nabla\varphi)=1$ a.e. in $kY$ and that $\det(F+\nabla\varphi_{n_j})\to\det(F+\nabla\varphi)$ in $L^1(kY)$. For the lower order minors of $F+\nabla\varphi_{n_j}$ we have at least weak convergence in $L^1(kY)$ to the fact that $p>d-1$. By the monotonicity of $\overline W_n^{(k)}(F)$ with respect to $n$ we have
	\begin{equation*}
		\sup_n \overline W_n^{(k)}(F)=\lim_{j\to\infty}\overline W_{n_j}^{(k)}(F)
		\geq \liminf_{j\to\infty}\fint_{kY}\min\left\{\widetilde W(y,F+\nabla \varphi_{n_j}),n_j(|F+\nabla\varphi_{n_j}|^p+1)\right\}\dx.
	\end{equation*}
	Since in general the minimum of two polyconvex functions is no longer polyconvex, we estimate the integrand from below by its polyconvex hull, i.e., the greatest polyconvex function below the map $F\mapsto V_{n_j}(x,F):=\min\{\widetilde{W}(x,F),n_j(|F|^p+1)\}$. Since $V_{n_j}$ is non-negative and finite, it follows from \cite[Theorem 6.6]{Dac} that the polyconvex hull $\mathcal{P}V_{n_j}$ with respect to the second variable is again a Carath\'eodory-function. Hence, continuing the previous chain of inequalities and using the weak lower semicontinuity of non-negative convex functionals together with the convergence properties of the minors of $F+\nabla\varphi_{n_j}$ derived above, for every $\bar n\in\N$ we have
	\begin{equation}\label{eq:withenvelope}
		\sup_{n}\overline{W}_n^{(k)}(F)\geq \fint_{kY}\mathcal{P}V_{\bar n}(x,F+\nabla\varphi)\dx.
	\end{equation}
	Clearly the sequence $\mathcal{P}V_{n}$ is monotone increasing with respect to $n$. We claim that it converges to $\widetilde{W}$ as $n\to +\infty$. To this end, we note that due the polyconvexity and non-negativity of $\widetilde{W}(x,\cdot)$ and \cite[Theorem 6.36]{FoLe} we find a sequence of polyaffine functions $a_i(x,\cdot)$ that approximates $\widetilde{W}(x,\cdot)$ from below and such that its coefficients belong to $L^{\infty}(kY)$. In particular, there exist constants $C_i\geq 1$ such that
	\begin{equation*}
		|a_i(x,F)|\leq C_i(|F|^d+1),
	\end{equation*}
	which then implies for every $i\in\N$ that
	\begin{equation*}
	\sup_n\mathcal{P}V_{n}(x,F)\geq \sup_n\mathcal{P}\min\{a_i(x,\cdot),n(|\cdot|^p+1)\}(F)=\mathcal{P}a_i(x,\cdot)(F)=a_i(x,F).
	\end{equation*}
	Letting $i\to +\infty$ we deduce that $\sup_n\mathcal{P}V_n(x,F)\geq \widetilde{W}(x,F)$. The other inequality is evident. Hence, letting $\bar n\to +\infty$ in \eqref{eq:withenvelope}, we deduce from monotone convergence that
	\begin{equation*}
		\sup_{n\in\N}\overline{W}_n^{(k)}(F)\geq\fint_{kY}\widetilde{W}(x,F+\nabla\varphi)\dx=\fint_{kY}W(x,F+\nabla\varphi)\dx\geq \overline{W}^{(k)}(F).
	\end{equation*}
	This concludes the proof.
\end{proof}

\begin{remark}[$\underline W(F)=W_{\rm hom}$?] A combination of Theorem~\ref{thm:upperbound} and Proposition~\ref{thm:lowerbound} yields a complete $\Gamma$-convergence result provided $\underline W=W_{\rm hom}$. The construction directly implies  $\underbar{W}\leq W_{\rm hom}$ but there are two serious obstructions to show the reverse inequality:

	%
	\begin{itemize}
		\item[i)] We would need to  exchange $\inf$ and $\sup$ and thus show commutation of truncation and homogenization. This is known to be true in a convex setting (see e.g.\ \cite{DG_unbounded,Mueller87}), but these arguments do not extend to a  non-convex setting. However, equation \eqref{claim:lb:whomk} shows that if the infima with respect to $k$ in the multi-cell formulas $\overline W_n$ and $\overline{W}$ are attained by a finite $k\in\N$, then truncation and homogenization commute. To the best of our knowledge there are no rigorous results on the relation between multi-cell and finite-cell formula in the context of incompressible elasticity. In the compressible case the example of M\"uller \cite{Mueller87} shows that in general the multi-cell formula does not reduce to a finite-cell formula, but in \cite{NS18,NS19} the equality of the multi-cell formula is proven for certain $F$ corresponding to small strains. 
		\item[ii)] In $W_{\rm hom}^{(k)}$ we take the infimum over Lipschitz-functions (see \eqref{eq:defW_hom}) while in $\overline W^{(k)}$ we take the infimum over $W^{1,p}$-functions (see \eqref{claim:lb:whomk}). While this makes no difference on the level of the approximation, that is, in \eqref{whomn}, the unboundedness of $W$ might feature the Lavrentiev phenomenon and thus we might have $\overline W^{(k)}(F)<W_{\rm hom}^{(k)}(F)$ or $\overline{W}(F)<W_{\rm hom}(F)$ for some $F\in\Sigma$.
	\end{itemize}
\end{remark}
\section*{Acknowledgments}
The research of MR is funded by the Deutsche Forschungsgemeinschaft (DFG, German Research Foundation) -- Project n$^\circ$ 530813503.

\end{document}